\documentclass{article}

\RequirePackage[T1]{fontenc}
\usepackage{microtype}

\usepackage{amsthm,amssymb}
\newtheorem{theorem}{Theorem}[section]

\newtheorem{lemma}[theorem]{Lemma}
\newtheorem{proposition}[theorem]{Proposition}

\theoremstyle{remark}
\newtheorem{remark}[theorem]{Remark}

\usepackage{amsmath,latexsym,bm}
\usepackage[numbers]{natbib}
\usepackage{enumitem}
\usepackage{mathtools}

\newcommand{\be}{\begin{eqnarray*}}
\newcommand{\bel}{\begin{eqnarray}}
\newcommand{\ee}{\end{eqnarray*}}
\newcommand{\eel}{\end{eqnarray}}
\newcommand{\ba}{\begin{aligned}}
\newcommand{\ea}{\end{aligned}}

\newcommand{\pa}{\partial}

\newcommand{\grad}{\nabla}
\newcommand\N{{\mathbb N}}
\newcommand\R{{\mathbb R}}
\newcommand\T{{\mathbb T}}

\newcommand{\RR}{\mathbb{R}}
\newcommand{\ZZ}{\mathbb{Z}}
\newcommand{\NN}{\mathbb{N}}
\newcommand{\TT}{\mathbb{T}}

\usepackage{xspace}

\makeatletter
\newcommand{\GI@st}[1]{\nonscript\:#1\vert\nonscript\:\mathopen{}\allowbreak}
\newcommand{\st}[1][]{\GI@st{#1}}
\DeclarePairedDelimiterX\set[1]\{\}{%
  \renewcommand{\st}{\GI@st{\delimsize}}#1%
}
\makeatother
\DeclarePairedDelimiter{\norm}{\lVert}{\rVert}

\DeclarePairedDelimiter{\paren}{(}{)}

\renewcommand{\leq}{\leqslant}
\renewcommand{\geq}{\geqslant}

\newcommand{\loc}{\operatorname{loc}}
\newcommand{\dv}{\nabla\cdot}
\newcommand{\one}{\bm{1}}

\usepackage[final,colorlinks]{hyperref}
\usepackage{doi}

\begin{document}

\title{Growth of Sobolev norms and loss of regularity in transport equations%
\footnotetext{%
  G. Crippa was partially supported by the ERC Starting Grant 676675 FLIRT.
  The remaining authors were partially supported by the US National Science Foundation through grants
    DMS 2043024 and DMS 2124748 to T. Elgindi,
    DMS 1814147 and DMS 2108080 to G. Iyer,
    and
    DMS 1909103 to A. Mazzucato.
}}

\author{
Gianluca Crippa\thanks{
  Department of Mathematics and Computer Science, University of Basel, Spiegelgasse 1, 4051 Basel, Switzerland}
\and
Tarek Elgindi\thanks{
  Mathematics Department, Duke University, 120 Science Drive, Durham, NC 27708-0320, U.S.A.}
\and
Gautam Iyer\thanks{
Department of Mathematical Sciences, Carnegie Mellon University, Pittsburgh, PA 15213, U.S.A.}
\and
Anna L. Mazzucato\thanks{
Mathematics Department, Penn State University, University Park, PA, 16802, U.S.A.}
}





\maketitle

\begin{abstract}
We consider transport of a passive scalar advected by an irregular divergence free vector field.
Given any non-constant initial data~$\bar \rho \in H^1_\text{loc}(\R^d)$, $d\geq 2$, we construct a divergence free advecting velocity field~$v$ (depending on $\bar \rho$) for which the unique weak solution to the transport equation does not belong to $H^1_{\loc}(\R^d)$ for any positive positive time.
The velocity field~$v$ is smooth, except at one point, controlled uniformly in time, and belongs to almost every Sobolev space $W^{s,p}$ that does not embed into the Lipschitz class.
The velocity field $v$ is constructed by pulling back and rescaling an initial data dependent sequence of sine/cosine shear flows on the torus.
This loss of regularity result complements that in {\em  Ann. PDE}, 5(1):Paper No. 9, 19, 2019.
\end{abstract}


\begin{centering}

\bigskip

{\em In memory of Charles ``Charlie'' Doering.}


\end{centering}

\section{Introduction} \label{s:intro}

\typeout{\linewidth=\the\linewidth}
This article concerns the effect of transport by an irregular vector field on a passive scalar. In what follows, we refer to {\em irregular transport} as transport by a vector field that does not possess Lipschitz regularity in the space variable.




 It is well known that, if the advecting vector field is Lipschitz uniformly in time, the Cauchy-Lipschitz theory applies and the flow is well-defined pointwise in space and time. The flow and its inverse are then also Lipschitz and Lipschitz regularity of the initial data is preserved under the action of the flow.
 In this case, the unique solution to the linear transport equation is obtained by composing the initial data with the inverse of the flow map.

 The regularity of weak solutions to the transport equation with an irregular advecting velocity field has been extensively studied by many authors (see for instance~\cite{DiPernaLions89,Ambrosio04,LeBrisLions04,CDL08.2,AmbrosioCrippa14,BrueNguyen21,BrueNguyen21a}).
 In this work, we are interested in loss of regularity for the weak solution of the transport equation, when the advecting vector field is ``almost Lipschitz'' in space.
 Our main result constructively shows that for any non-constant, $H^1_\text{loc}(\R^d)$ initial data there is a bounded, compactly supported, divergence-free vector field, which  is (uniformly in time) almost Lipschitz in space, such that the solution to the associated irregular transport equation loses its $H^1$ regularity instantaneously.
More precisely, fix any (non-constant) initial data in $H^1_\text{loc}(\RR^d)$, with $d\geq 2$.
We produce a bounded, compactly supported, divergence-free vector field, depending on the initial data,  which is (uniformly in time) in the Sobolev space $W^{r, p}$, for every $r \geq 0$ and $p \in [1, \infty)$ such that $r < 1 + d/p$.
(We recall $r = 1 + d/p$ is the critical threshold for the Sobolev embedding, threshold  above which $W^{r,p}$ embeds into the Lipschitz space $W^{1, \infty}$.)
Moreover, the vector field is constructed so that the solution of the associated transport equation is not in $H^1_\text{loc}$ for any $t > 0$.
The loss of regularity is due to an amplification effect on the derivative of the solution by the action of the advecting flow.

To fix notation, we denote the passive scalar by $\rho=\rho(x,t)$, with $t\geq 0$, $x=(x_1,x_2, \ldots, x_d)\in \RR^d$, and the advecting field by $v=v(x,t)$.
We assume $\rho$ is a weak solution of the linear transport equation:
\begin{equation} \label{eq:transport}
  \pa_t \rho + v\cdot\nabla \rho=0,
\end{equation}
on $\RR^d\times [0,\infty)$, with initial data $\bar\rho(x)$.


The function spaces mentioned above follow standard notation.
Namely, for $k\in \ZZ_+$ and $1\leq p\leq \infty$, the space $W^{k,p}(\RR^d)$ is the Sobolev space defined by
\[
    W^{k,p}(\R^d)=\{ f\in L^p(\RR^d)\;\mid\; \pa^\alpha f\in L^p(\RR^d), \; |\alpha|\leq k \},
\]
where we have used the multi-index formalism for derivatives.
When $r > 0$ is not an integer $W^{r,p}$ denotes the fractional Sobolev space which is defined by interpolation (see~\cite{Adams} for a comprehensive introduction).
For $p=2$, the space $W^{r,2}$ coincides with the space $H^{r}$, defined via the Fourier Transform.
\smallskip

The loss of regularity result presented here extends the results by some of the authors in \cite{ACMReg19}. There, it was proved that there exists a smooth, compactly supported initial data $\Bar\rho$ and a vector field $v\in L^\infty([0,\infty);W^{1,p}(\R^d))$, for $1< p <\infty$, such that the weak solution $\rho$ of \eqref{eq:transport} does not belong to $H^s$ for any $s>0$ instantaneously in time.
In contrast, we are able to show loss of regularity for all non-constant initial data in $H^1_{\text{loc}}$ (with $v$ depending on the initial data), but we can only prove $\rho(\cdot,t)\notin H^s_{\text{loc}}$, for any $s\geq 1$ and for all $t>0$.
(We also mention that in~\cite{GG20}, the authors prove, non-constructively,  that loss of regularity is a generic phenomenon in the sense of Baire's Category Theorem.)

In both~\cite{ACMReg19} and this work, we construct at the same time the vector field $v$ and the advected scalar
$\rho$ via an iterative procedure starting from a pair  $u^0$, $\theta^0$ (where $\theta^0$ solves the transport equation with advecting field $u^0$)
which acts as a building block, and applying a suitable sequence of rescalings, where each rescaling produces a pair $u^n$, $\theta^n$.   In \cite{ACMReg19}, $u^0$ is a vector field that mixes a certain initial tracer configuration optimally in time, and one can control the growth of the $H^s$ norm of $\theta^0$ from below for all $s>0$ via interpolation, since $u^0$ drives all negative Sobolev norms of the tracer to zero exponentially fast.
The action of each rescaling is to accelerate the growth of the $H^s$-norms of $\theta^n$ as $n\to \infty$. The different $u^n$ and $\theta^n$ are combined to give rise to the vector field $v$ and associated weak solution $\rho$ of \eqref{eq:transport}, the Sobolev norms of which blow up for any $t>0$. This result is optimal from the point of view of the loss of regularity, in the sense that the only regularity that is propagated generically by a velocity field with the same regularity as $v$ is essentially a ``logarithm'' of a derivative \cite{BN20,CDL08}. We mention also the related work \cite{Jab16}, where the author gives an example of a divergence-free vector field in $H^1$ such that its flow is not in any Sobolev space with positive regularity. His construction is random at its core, while the one in \cite{ACMReg19} is deterministic and explicit.

In this note, we also use a suitable sequence of rescalings of basic flows. These flows are constructed in such a way to lead to growth in time of the $H^1$ Sobolev norm of any initial data for the passive scalar. Although the vector field depends on the initial data, it enjoys universal bounds. The vector fields are constructed using shear flows and, after rescaling, the growth happens on certain cubes that depend on the initial data $\Bar\rho$ for \eqref{eq:transport}.

Our main result is the following.

\begin{theorem} \label{t:mainThm}
 Let $\bar\rho\in H^1_{\text{loc}}(\RR^d)$, $d\geq 2$, be a non-constant function.
 There exists a compactly supported divergence-free vector field $v\in L^\infty([0,\infty) \times \RR^d)$, depending on $\bar\rho$, such that the following hold:
 \begin{enumerate}[label=(\alph*),ref=(\alph*),topsep=\smallskipamount,itemsep=\smallskipamount]
  \item\label{i:mainThm1}
    The velocity field~$v$ is smooth except at one point in $\RR^d$.
    Moreover,
    \begin{equation*}
      v \in L^\infty([0,\infty);W^{r,p}(\RR^d))
      \qquad\text{for every }
      p \in [1, \infty)\,,
      \text{ and }
      r < \frac{d}{p} + 1\,.
    \end{equation*}

  \item The unique weak solution of~\eqref{eq:transport} in  $L^\infty([0,\infty);L^2_{\text{loc}}(\RR^d))$ with initial data $\Bar\rho$ is such that
    \begin{equation*}
      \rho(\cdot, t) \not \in H^1_{\text{loc}}(\R^d)
      \qquad
      \text{for every }
      t > 0 \,.
    \end{equation*}
 \end{enumerate}
\end{theorem}

As mentioned earlier, if $r > d/p + 1$ and $v \in L^\infty( [0, \infty); W^{r, p}(\R^d))$, then the Sobolev embedding theorem implies that~$v$ is Lipschitz in space, uniformly in time.
This in turn implies that~$H^1$ regularity of the initial data is preserved and so the threshold $r < d/p + 1$ above can not be improved.

The main idea behind the proof is as follows:
\begin{enumerate}[label=(\arabic*),topsep=\smallskipamount,itemsep=\smallskipamount]
  \item
    The first step is an elementary observation about periodic functions.
    Take any non-constant periodic function~$\bar \phi$.
    Then, we claim at least one sine or cosine shear flow parallel to one of the coordinate axis must increase the~$H^1$ norm of~$\bar\phi$ by a constant factor (see Lemma~\ref{l:shearsTorus}, below).

  \item
    By localizing and rescaling the above flow, we can obtain a countable (shrinking) family of separated cubes that cluster at one point, so that in each cube the flow increases the $H^1$ norm of the advected scalar by a larger and larger factor (see Section~\ref{s:covering}, below).

  \item 
    Now we need to ensure that the rescaling factors and the location of the cubes can be chosen so that the $H^1$ norm of the solution diverges at any positive time, but the velocity field remains sufficiently regular.
    Our choice  ensures $v \in W^{r, p}$ for every $r$ below the critical Sobolev embedding threshold (i.e.\ $r < d/p + 1$).
\end{enumerate}

The rest of the paper is organized as follows. In Section \ref{s:torus}, we introduce the basic building block in the construction and  show how the building block leads to growth of the Sobolev norms for solutions of the transport equation \eqref{eq:transport}. Then, in Section \ref{s:covering} we conclude the proof of loss of regularity. Lastly, in Section \ref{s:end} we draw some conclusions.

Throughout the paper, we denote the total mass of any measurable (with respect to the $d$-dimensional Lebesgue measure) set $\Omega$ by $|\Omega|$, while $\mathbf{1}_{\Omega}$ denotes the indicator function of the set $\Omega$.
The symbols $\lesssim, \gtrsim$ denote bounds which holds up to a generic constant that may change from line to line.

\section{Construction of the basic flow and growth of Sobolev norms}\label{s:torus}

The aim of this section is to carry out the first step in the proof of the main theorem.
We first prove the elementary observation (Lemma~\ref{l:shearsTorus}, below) that for any non-constant periodic function, at least one sine or cosine shear along a coordinate axis can be used to increase its $H^1$ norm by a constant factor.
Next we lift this construction to compactly supported cubes in $\R^d$, and iterate to obtain  exponential growth in time (Proposition~\ref{p:expgrowth}, below).
This will be the basic building block that will be rescaled and used in subsequent steps in Section~\ref{s:covering}. 

To notationally separate the construction of our building block from the actual rescaled flow in Theorem~\ref{t:mainThm}, in this section we use~$u$ to denote the advecting velocity field on the torus and $\phi$ to denote the passively advected (periodic) scalar with initial data~$\bar \phi$.
For convenience we will work with $8$-periodic functions on the $d$-dimensional torus~$\T^d$ obtained by identifying parallel faces of the cube $[0, 8]^d$.

\begin{lemma}\label{l:shearsTorus}
  Let $A > 0$ and define $f_1, f_2 \colon \R \to \R$ by
  \begin{equation*}
    f_1(z) = A \sin( 2\pi z) \qquad\text{and}\qquad f_2(z) = A \cos(2\pi z)\,,
  \end{equation*}
  and let $\Omega_0 \subseteq \T^d$, $d\geq 2$, be a piecewise $C^1$ domain.
  For any $\bar\phi \in H^1(\T^d)$, $T > 0$, there exists a divergence-free velocity field~$U$ (depending on $\one_{\Omega_0} \bar\phi$ and $T$) such that the following hold:
  \begin{enumerate}
    \item 
      The velocity field~$U$ is a shear flow of the form
      \begin{equation}\label{e:uj}
	U(x) = \pm f_i(x_j) e_{j'}\,,
	\qquad\text{where }
	j' = \begin{cases} j+1 & j < d\,\\
	  1 & j = d\,.
	\end{cases}
      \end{equation}
      Here $e_j \in \R^d$ is the $j^\text{th}$ standard basis vector, and $x_j$ denotes the $j^\text{th}$ coordinate of $x \in \T^d$.

    \item 
      The solution to the transport equation
      \begin{equation} \label{eq:transportU}
	   \partial_t \phi +U\cdot \nabla \phi =0
      \end{equation}
      on $\TT^d$ with initial data~$\bar\phi$ satisfies
      \begin{equation}\label{e:gradThT}
	\norm{\grad \phi(\cdot, T)}_{L^2(\Omega_T)}^2 \geq \paren[\Big]{1 + \frac{2 \pi^2 A^2 T^2}{d} } \norm{\grad \bar\phi}_{L^2(\Omega_0)}^2\,.
      \end{equation}
      Here\ $\Omega_T$ is the image of $\Omega_0$ under the flow map of the shear flow~$U$ after time~$T$.
  \end{enumerate}
\end{lemma}
\begin{proof}
  Given $i, i' \in \set{1, 2}$ and $j \in \set{1, \dots, d}$, we  let
  \begin{equation*}
    u_{i,i', j}(x) = (-1)^i f_{i'}(x_j) e_{j'},
  \end{equation*}
   and we let~$\phi_{i,i',j}$ be the solution of the transport equation~\eqref{eq:transportU} with vector field $u_{i,i',j}$. We denote by $\Omega_{T, i, i', j}$ the image of $\Omega_0$ under the flow map of the shear flow $u_{i,i',j}$ after time~$T$.
  Since
  \begin{equation*}
    \phi_{i, i', j}(x, t)
      = \bar\phi( x - (-1)^{i} f_{i'}(x_j) \,t \,e_{j'} )\,,
  \end{equation*}
  we compute
  \begin{equation*}
    \partial_k \phi_{i, i', j}
      = \begin{cases}
	  \partial_k \bar \phi - (-1)^i f_{i'}(x_j) t \partial_{j'} \bar\phi  & k = j\,,
	  \\
	  \partial_k \bar \phi & k \neq j\,.
      \end{cases}
  \end{equation*}
  We square the expression above and  sum over $i, i'$. Using the fact that $\sum_{i'} f_{i'}^2 = A^2$, integrating over $\Omega_{T,i,i',j}$, and changing variables back to the original domain $\Omega_0$ gives
  \begin{equation*}
    \sum_{i, i'} \norm{\partial_k \phi_{i, i', j} }_{L^2(\Omega_{T, i, i', j})}^2
    = \begin{cases}
      4 \norm{\partial_j \bar\phi}_{L^2(\Omega_0)}^2 + 8 \pi^2 A^2 T^2 \norm{\partial_{j'} \bar \phi}_{L^2(\Omega_0)}^2 & k = j\,.
      \\
      4 \norm{\partial_k \bar\phi}_{L^2(\Omega_0)}^2 & k \neq j\,,
    \end{cases}
  \end{equation*}
  Summing over~$k \in \set{1, \dots, d}$ and $j \in \set{1, \dots, d}$ then shows that
  \begin{equation*}
    \sum_{i, i', j} \norm{\grad \phi_{i, i', j}}_{L^2(\Omega_{T, i, i', j})}^2
    = 4 d \,\norm{\grad \bar\phi}_{L^2(\Omega_0)}^2 + 8 \pi^2 A^2 T^2 \norm{\grad \bar \phi }_{L^2(\Omega_0)}^2\,.
  \end{equation*}
  Since there are $4d$ terms on the sum on the left, there must exist one term that is  at least a $1/(4d)$ fraction of  the right hand side.
  This immediately yields~\eqref{e:gradThT} as claimed.
\end{proof}

Our next task is to show that for any (non-constant) initial datum, we can find a smooth compactly supported divergence-free vector field in $\R^d$ for which the solution to the transport equation grows exponentially in $H^1$.
This is the main result of this section, and is what will be used in the proof of Theorem~\ref{t:mainThm}.

\begin{proposition}\label{p:expgrowth}
  Let  $\bar\theta \in H^1_{\loc}(\R^d)$, $d\geq 2$, and fix $\alpha > 0$.
  There exist a constant $C(\alpha, d)$ (independent of~$\bar \theta$) and a divergence-free vector field $u \colon \RR^d\times [0,\infty) \to \R^d$ (depending on $\bar\theta$) such that $u$ is piecewise constant in time, supported on the cube $\tilde \Omega_0=(-3,4)^d$, satisfies the bound
  \begin{equation*}
   \sup_{0 \leq \tau < \infty} \norm{u(\cdot, \tau)}_{C^1(\R^d)} \leq C(\alpha,d)\,,
  \end{equation*}
  and the following two assertions hold.
  \begin{enumerate}
    \item
      The unique solution of the transport equation
      \begin{equation}\label{e:transThetaU}
	\partial_t \theta + u \cdot \grad \theta = 0
      \end{equation}
      in $\RR^d$ with initial data $\bar\theta$, satisfies
      \begin{equation*}
	\norm{\grad \theta(\cdot, n)}_{L^2(\Omega_0)}
	  \geq e^{\alpha n} \norm{\grad \bar\theta}_{L^2(\Omega_0)}\,,
      \end{equation*}
      for all non-negative integer times $n \in \N$.
      Here $\Omega_0$ is the cube $(0, 1)^d$ in $\RR^d$;
    \item
      For all times $t \geq 0$, the above solution~$\theta$  satisfies
      \begin{equation}\label{e:expGrowth}
	\norm{\grad \theta(\cdot, t)}_{L^2(\tilde \Omega_0)}
	  \geq e^{\alpha t - \beta} \norm{\grad \bar\theta}_{L^2(\Omega_0)} \,.
      \end{equation}
      Here~$\beta$ is a constant that depends on~$\alpha$ and $d$, but not on~$\bar\theta$.
  \end{enumerate}
\end{proposition}
\begin{remark} \label{r:uSmooth}
  With minor modifications to the proof one can ensure that the velocity field~$u$ in Proposition~\ref{p:expgrowth} is in fact smooth, and satisfies $\norm{u(\cdot, t)}_{C^k} \leq C(\alpha, d, k)$ for all $t \geq 0$.
\end{remark}

The proof of Proposition~\ref{p:expgrowth} consists of two steps.
The first step involves pulling back the shear flow on the torus from Lemma~\ref{l:shearsTorus} to a compactly supported flow in $\R^d$.
We do this in Lemma~\ref{l:growth}, below.
Once this is established, we simply iterate this procedure to obtain exponential growth at integer times.
Since the norm of~$u$ is controlled uniformly in time, the $H^1$ norm at non-integer times can be estimated by giving up a small factor.

\begin{lemma}\label{l:growth}
 Let  $\bar\theta \in H^1_{\loc}(\R^d)$, $d\geq 2$, and fix $T > 0$, $\alpha' > 1$.
 There exists a divergence-free vector field $u$ on $\RR^d\times [0,\infty)$ (depending on $\bar\theta$, $\alpha'$ and~$T$) such that the following hold:
 \begin{enumerate}
   \item The vector field~$u$ is piecewise constant in time, supported on the cube $\tilde \Omega_0=(-3,4)^d$, and satisfies
      \[
       \sup_{0 \leq \tau \leq T} \norm{u(\cdot, \tau)}_{C^1(\R^d)} \leq C(d) \paren[\Big]{ 1 + \frac{\alpha'}{T} } \,,
      \]
      for some dimensional constant $C(d)>0$, that is independent of $\bar\theta$.
  \item 
    The weak solution of the transport equation~\eqref{e:transThetaU}
    in $\RR^d$ with initial data $\bar\theta$ satisfies
    \begin{equation*}
	\norm{\grad \theta(\cdot, T)}_{L^2(\Omega_0)}
	  \geq \alpha' \, \norm{\grad \bar\theta}_{L^2(\Omega_0)}\,,
      \end{equation*}
    where $\Omega_0 = (0, 1)^d \subseteq \R^d$.
 \end{enumerate}
\end{lemma}

\begin{figure}[htb]%
  \noindent\hfil%
  \begin{minipage}{.4\linewidth}
    \centering
    \includegraphics[width=\linewidth]{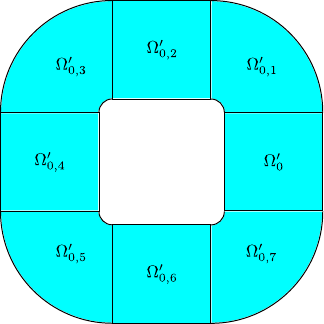}
    \caption{The rounded octagonal track~$\mathcal A_1'$}
    \label{f:octTrack}
  \end{minipage}%
  \hfil%
  \begin{minipage}{.4\linewidth}
    \centering
    \includegraphics[width=\linewidth]{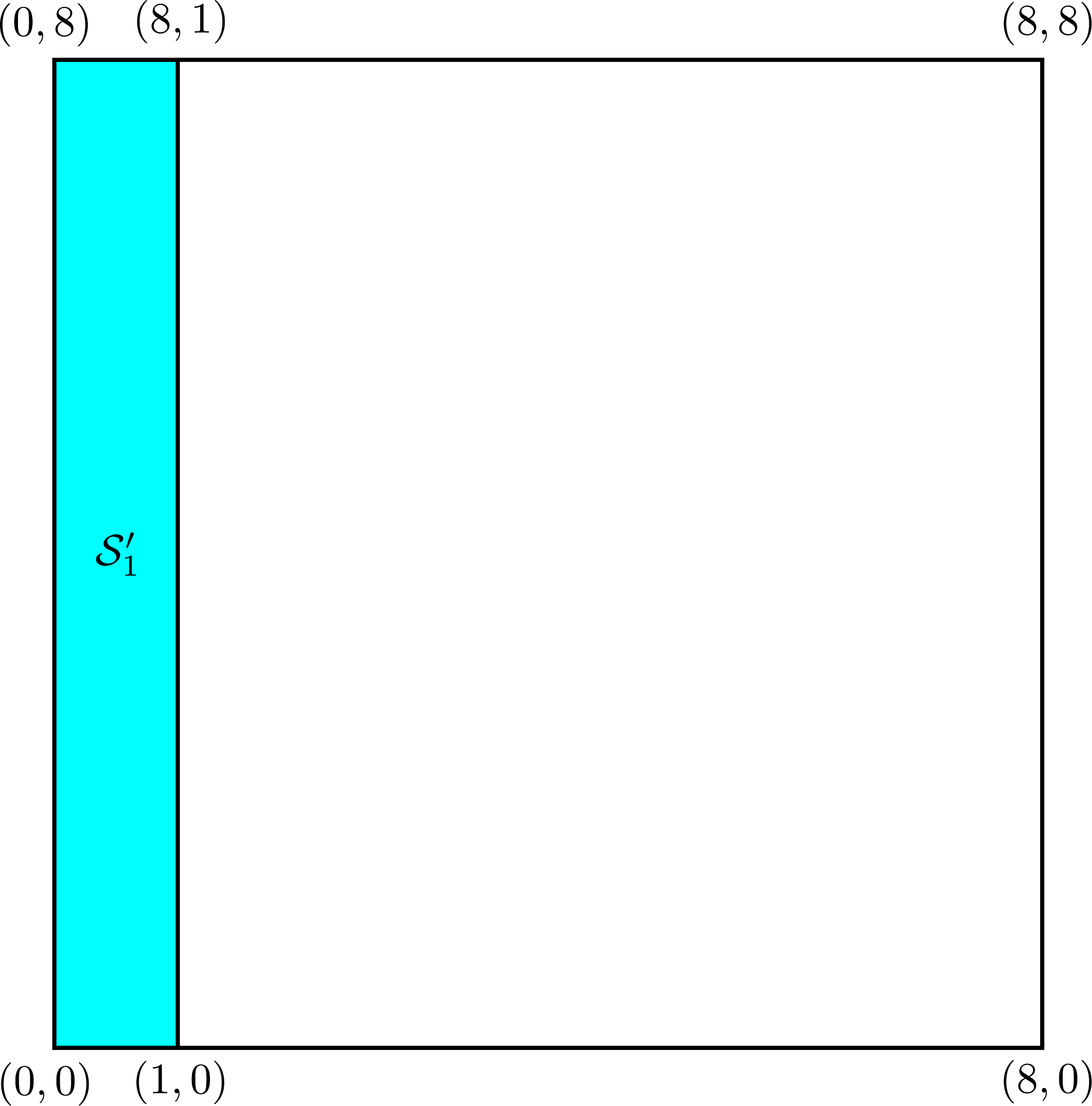}
    \caption{The strip $\mathcal S_1' \subseteq \T^2$.}
    \label{f:strip}
  \end{minipage}%
  \hfil
\end{figure}
The main idea behind the proof of Lemma~\ref{l:growth} is as follows.
Momentarily suppose $d = 2$ and view~$\Omega_0$ as a subset of the two-dimensional torus $\T^2$ obtained by identifying parallel sides of the square $[0, 8]^2$.
Now, by Lemma~\ref{l:shearsTorus}, there is a horizontal or vertical trigonometric shear, $U$, that increases the~$H^1$ norm by a constant factor.
Suppose this shear was vertical.
In this case the flow would spread out the initial data over the vertical strip~$\mathcal S_1'$, shown in Figure~\ref{f:strip}.
The strip $\mathcal S_1' \subseteq \T^2$ is topologically an annulus, and so we can find an annulus $\mathcal A_1' \subseteq \R^2$ (see Figure~\ref{f:octTrack}) and an area preserving diffeomorphism $\mathcal \varphi_1 \colon \mathcal A_1' \to S_1'$ such that $\varphi_1$ is the identity on~$\Omega_0$.
We use~$\varphi_1$ to pullback~$U$ to a vector field~$u$ on~$\mathcal A_1'$.
This velocity field will spread the initial data out in the track~$\mathcal A_1'$.
However, since the area of $\Omega_0$ is one eighth the area of~$\mathcal A_1'$, one can give up a factor of~$8$, perform a radial rotation along the track and ensure that the $H^1$ norm in~$\Omega_0$ itself grows as desired.
We now carry out the details.

\begin{proof}[Proof of Lemma~\ref{l:growth}]
  Let $\mathcal A_1' \subseteq \R^2$ be the rounded octagonal track constructed as follows (see Figure~\ref{f:octTrack}): the region $\Omega_0'$ is the square $(0, 1)^2 \subseteq \R^2$, the regions $\Omega_{0, 2}'$, $\Omega_{0, 4}'$ and $\Omega_{0, 6}'$ are squares of side length~$1$.
  The remaining four regions are quarter annuli with inner radius $\frac{2}{\pi} - \frac{1}{2}$ and outer radius $\frac{2}{\pi} + \frac{1}{2}$.
  These radii are chosen so that the area of each piece is~$1$. We observe that $\mathcal A_1' \subset (-3,4)^2$.

  Let~$\mathcal S_1' = (0, 1) \times (0, 8) \subseteq \T^2$ be the strip of width~$1$ parallel to the~$x_2$ axis (see Figure~\ref{f:strip}).
  Let $\varphi_1' \colon \bar{\mathcal A}_1' \to \bar{\mathcal S}_1'  \subseteq \T^2$ be an area preserving diffeomorphism such that
  \begin{equation*}
    \varphi_1'(x') = x' \qquad\text{for all } x' \in \Omega_0'\,.
  \end{equation*}
  This map can be explicitly constructed by simply deforming each of the quarter annuli into unit squares, and performing the appropriate rotation on each of the squares $\Omega_{0, 2}'$, $\Omega_{0, 4}'$ and $\Omega_{0, 6}'$.

  In $d$-dimensions, we define~$\mathcal A_1 = \mathcal A_1' \times (0, 1)^{d-2} \subseteq \R^d$, and $\mathcal S_1 = \mathcal S_1' \times (0, 1)^{d-2} \subseteq \T^d$. We observe that $\mathcal A_1 \subset (-3,4)^d$.
  We define $\varphi_1 \colon \bar{\mathcal A}_1 \to \bar{\mathcal S}_1$ by
  \begin{equation*}
    \varphi_1(x_1, \dots, x_d) = (\varphi_1'(x_1, x_2), x_3, \dots, x_d)\,,
  \end{equation*}
  and note that $\varphi_1(x) = x$ for all~$x \in (0, 1)^d$.
  Finally, for each~$j \in \set{2, \dots, d-1}$ we repeat the above procedure along the $j^\text{th}$ and $(j+1)^\text{th}$ axis, and for $j = d$ we do the same along the $j^\text{th}$ and $1^\text{st}$ axis.
  This yields the regions~$\mathcal A_j$, and corresponding maps~$\varphi_j \colon \bar{\mathcal A}_j \to \T^d$.

  Now, we let $\bar \phi$ be an~$H^1$ extension of~$(\one_{\Omega_0} \bar\theta) \circ \varphi_1^{-1}$ to $\T^d$.
  We note that our choice of $\varphi_j$ implies $(\one_{\Omega_0} \bar\theta) \circ \varphi_1^{-1} = (\one_{\Omega_0} \bar\theta) \circ \varphi_j^{-1}$ for all $j \in \set{1, \dots, d}$. Let~$A > 0$ be a large constant that will be chosen shortly. By Lemma~\ref{l:shearsTorus} there exist~$j \in \set{1, \dots, d}$ and a shear flow $U$ on~$\T^d$, directed along the $j'$-th coordinate axis, such that $U$ is the form~\eqref{e:uj} and
  \begin{equation*}
    \norm{\grad \phi(\cdot, T)}^2_{L^2(\Omega_T)}
    \geq \paren[\Big]{ 1 + \frac{2 \pi^2 A^2 T^2}{d} } \norm{\grad \bar \phi}^2_{L^2(\Omega_0)}\,.
  \end{equation*}
  Here~$\phi$ is the solution of the transport equation~\eqref{eq:transportU} on $\T^d$ with initial data~$\bar \phi$.
  For simplicity, and without loss of generality, we will now assume~$j = 1$.

  Next, we let~$\tilde{u} \colon \mathcal A_1 \to \R^d$ be the pullback of~$U$ under~$\mathcal \varphi_1$.
  That is, we define
  \begin{equation*}
    \tilde{u} = (\varphi_1^{-1})^*(U) = (D \varphi_1^{-1} U) \circ \varphi_1\,.
  \end{equation*}
  Since~$\varphi_1$ preserves the Lebesgue measure, and $\dv U = 0$ we must also have~$\dv \tilde{u} = 0$.
   Now extend~$\tilde{u}$ to be a~$C^1$ divergence-free vector field supported in~$(-3,4)^d$, and let~$\tilde \theta$ be the solution to the transport equation
  \begin{equation*}
    \partial_t \tilde\theta + \tilde{u} \cdot \grad \theta = 0
  \end{equation*}
  in $\R^d$ with initial data~$\bar \theta$.
  By the construction of~$\tilde{u}$ and the fact that~$\bar \theta = \bar \phi \circ \varphi_1$ on~$\Omega_0$, we must have
  \begin{equation*}
    \tilde\theta(x, t) = \phi(\varphi_1(x), t)
    \qquad\text{for all } x \in \Omega_t\,,
  \end{equation*}
  where~$\Omega_t$ is the image of~$\Omega_0$ under the flow map of~$\tilde{u}$ after time $t$.
  Hence,
  \begin{align}
    \norm{\grad \tilde\theta(\cdot, T)}_{L^2(\mathcal A_1)}^2
      &\geq
	\norm{\grad \varphi_1^{-1}}_{L^\infty}^{-2} \norm{\grad \phi(\cdot, T)}^2_{L^2(\mathcal S_1)}
      \geq \norm{\grad \varphi_1^{-1}}_{L^\infty}^{-2} \norm{\grad \phi(\cdot, T)}^2_{L^2(\Omega_T)}
       \nonumber \\
      &\geq
	\norm{\grad \varphi_1^{-1}}_{L^\infty}^{-2}
	\paren[\Big]{1 + \frac{2\pi^2 A^2 T^2}{d} }
	\norm{\grad \bar \phi}_{L^2(\Omega_0)}^2
      \geq
	\alpha_0'
	\norm{\grad \bar \theta}_{L^2(\Omega_0)}^2, \label{eq:A1Estimate}
  \end{align}
  where
  \begin{equation*}
    \alpha_0' = 
      \norm{\grad \varphi_1^{-1}}_{L^\infty}^{-2}
      \norm{\grad \varphi_1}_{L^\infty}^{-2}
      \paren[\Big]{1 + \frac{2 \pi^2 A^2 T^2}{d} }\,.
  \end{equation*}

  To finish the proof, we need to replace the left hand side of the above with $\norm{\grad \tilde\theta(\cdot, T)}_{L^2(\Omega_0)}$.
  To do this we divide~$\mathcal A_1$ into eight regions of equal measure, and note that on at least one of these regions we must have $\norm{\grad \tilde\theta(\cdot, T)}_{L^2(\Omega_{0, i})}^2 \geq \frac{1}{8} \norm{\tilde\theta(\cdot, T)}_{L^2(\mathcal A_1)}^2$.
  If we now use a flow, $\tilde w$, that shifts this region back to~$\Omega_0$, then we will have the desired inequality.
  We elaborate on this below.

  The flow~$\tilde w$ above can be constructed as follows:
  Let $U = -e_2$, and view $U$ as a flow on the strip~$\mathcal S_1 \subseteq \T^d$.
  Let $\tilde w$ be the pullback of $U_2$ under the map~$\varphi_1$.
  By construction of~$\varphi_1$ we note that for every $i \in \set{0, 7}$, the flow of~$\tilde w$ will map the region~$\Omega_{0, i}$ to the region $\Omega_{0, 0} = \Omega_0$ in time $i$.
  (Here $\Omega_{0, i} = \Omega_{0, i}' \times (0, 1)^{d-2} \subseteq \mathcal A_1$, where $\Omega_{0, i}'$ is shown in Figure~\ref{f:octTrack} and described at the beginning of the proof.)

  From \eqref{eq:A1Estimate}, there must exist $i \in \set{0, \dots, 7}$ such that
  \begin{equation*}
    \norm{\grad \tilde \theta(\cdot, T)}_{L^2(\Omega_{0, i})}^2
      \geq \frac{1}{8} \norm{\grad \tilde \theta(\cdot, T)}_{L^2(\mathcal A_1)}^2
      \geq \frac{\alpha_0'}{8} \norm{\grad \bar \theta}_{L^2(\Omega_0)}^2\,.
  \end{equation*}
  With this~$i$ we define the desired velocity field~$u$ by
  \begin{equation*}
    u(x, t) =
      \begin{dcases}
	\tilde u(x) & 0 \leq t \leq T\,,\\
	\tilde w(x) & T < t \leq T + i\,,
      \end{dcases}
  \end{equation*}
  and let~$\theta$ solve~\eqref{e:transThetaU} with initial data~$\bar \theta$.
  Notice $\theta(\cdot, t) = \tilde \theta( \cdot, t)$ for all $t \in [0, T]$, and
  \begin{equation*}
    \theta(x, T+i) =  \tilde \theta( \tilde \varphi^{-1}_{\tilde w}(x, i) )\,,
  \end{equation*}
  where $\tilde \varphi_{\tilde w}( \cdot, t)$ is the flow map of~$\tilde w$ after time~$t$.
  Consequently,
  \begin{align*}
    \norm{\grad \theta(\cdot, T+i)}_{L^2(\Omega_0)}
      &\geq \norm{\grad \varphi_1^{-1}}_{L^\infty}^{-2}
	\norm{\grad \varphi_1}_{L^\infty}^{-2}
	\norm{\grad \theta(\cdot, T)}_{L^2(\Omega_{0, i})}
    \\
      &\geq \norm{\grad \varphi_1^{-1}}_{L^\infty}^{-2}
	\norm{\grad \varphi_1}_{L^\infty}^{-2}
	\frac{\alpha_0'}{8} \norm{\grad \bar \theta}_{L^2(\Omega_0)}^2
      \geq \alpha' \norm{\grad \bar\theta}_{L^2(\Omega_0)}\,,
  \end{align*}
  provided we choose $A = \alpha' C(d) / T$, for some large dimensional constant $C(d)$ that only depends on $d$.
  Note that
  \begin{equation*}
    \sup_{0 \leq t \leq T+i} \norm{u}_{C^1}
      \leq \max\set{C_1(d) A, C_2(d)}
  \end{equation*}
  for some dimensional constants~$C_1(d)$ and $C_2(d)$.
  Thus rescaling time by a factor of $T / (T + i)$ the velocity field~$u$ satisfies all the conditions in the statement of Lemma~\ref{l:growth}.
  This concludes the proof.
\end{proof}

We conclude this section by repeatedly applying Lemma~\ref{l:growth} to prove Proposition~\ref{p:expgrowth}.
\begin{proof}[Proof of Proposition~\ref{p:expgrowth}]
  We first apply Lemma~\ref{l:growth} with~$T = 1$ and $\alpha' = e^\alpha$ to obtain a velocity field~$u$ such that
  \begin{equation*}
    \norm{\grad \theta(\cdot, 1)}_{L^2(\Omega_0)}
      \geq e^{\alpha} \norm{\grad \bar\theta}_{L^2(\Omega_0)}\,,
    \qquad\text{and}\qquad
    \sup_{0 \leq t \leq 1} \norm{u(\cdot, t)}_{C^1(\R^d)} \leq C(\alpha)\,.
  \end{equation*}
  Now we apply Lemma~\ref{l:growth} starting at time~$1$ with initial data~$\theta( \cdot, 1)$ to obtain a velocity field~$u$ (defined for $1 \leq t \leq 2$) such that
  \begin{equation*}
    \sup_{1 \leq t \leq 2} \norm{u(\cdot, t)}_{C^1(\R^d)} \leq C(\alpha)\,,
  \end{equation*}
  and
  \begin{equation*}
    \norm{\grad \theta(\cdot, 2)}_{L^2(\Omega_0)}
      \geq e^{\alpha} \norm{\grad \theta(\cdot, 1)}_{L^2(\Omega_0)}
      \geq e^{2\alpha} \norm{\grad \bar\theta}_{L^2(\Omega_0)}\,.
  \end{equation*}
  Note that the constant $C(\alpha)$ remained unchanged, since we only apply Lemma~\ref{l:growth} for a time interval of length~$1$.
  Proceeding inductively we obtain the first assertion in Proposition~\ref{p:expgrowth}.

  For the second assertion, we let $n \in \N$ and $t \in [n, n+1)$.
  Since the flow of the velocity field~$u$ preserves the domain~$\tilde \Omega_0$, and since \ $\sup_{0\leq t<\infty} \norm{u}_{C^1} \leq C(\alpha)$, we must have
  \begin{align*}
    \norm{\grad \theta( \cdot, t )}_{L^2(\tilde \Omega_0)}^2
      &\geq \frac{1}{C_1(\alpha)} \norm{\grad \theta(\cdot, n)}_{L^2(\tilde \Omega_0)}^2
      \\
      &\geq \frac{1}{C_1(\alpha)} \norm{\grad \theta(\cdot, n)}_{L^2(\Omega_0)}^2
      \geq \frac{e^{\alpha n}}{C_1(\alpha)} \norm{\grad \bar\theta}_{L^2(\Omega_0)}^2\,,
  \end{align*}
  for some constant $C_1(\alpha)$ that depends on~$\alpha$ but not~$\bar \theta$.
  This immediately implies the second assertion, finishing the proof.
\end{proof}

\section{Loss of regularity for the transport equation} \label{s:covering}

In this section we conclude the proof of Theorem~\ref{t:mainThm}. The basic idea of the proof resembles very closely that in~\cite{ACMReg19}, but with some important differences.

Both proofs entail an iterative construction in which some ``building block'' is replicated on a disjoint family of cubes at smaller spatial scales. The building block in~\cite{ACMReg19} is an optimal mixer from~\cite{ACMMix19}, which enjoys uniform-in-time bounds on the first-order derivatives and decreases the negative norms of a specific advected scalar exponentially in time. By interpolation, the positive norms of the scalar increase exponentially in time, and roughly speaking the iterative construction entails a rescaling in time that makes the exponential increase an instantaneous blow up, still keeping under control the $W^{1,p}$ norm of the vector field for every $p<\infty$. By contrast, in the present proof we rely on the velocity field constructed in Section~\ref{s:torus}, which increases the $H^1$ norm of the advected scalar exponentially in time, but in general it is not mixing.
The advantage of this approach is that higher regularity norms of the velocity field are controlled uniformly in time, and that the growth of the Sobolev norm holds for every (nontrivial) advected scalar with initial data in $H^1$.
We will therefore be able to keep under control higher $W^{r,p}$ norms of the vector field uniformly in time, and to show loss of $H^1$ regularity for every such initial data. In fact, since the construction is local, we need only assume that the initial data is locally in $H^1(\RR^d)$.

The iterative construction becomes however less explicit, since the location and the spatial scale of the family of cubes depend on the initial data, as we need to select the cubes in such a way that the derivative of the initial data is large enough in all of the cubes.

\begin{proof}[Proof of Theorem~\ref{t:mainThm}] We divide the proof in three steps. 

{\em Step 1. Set-up of the geometric construction.} We need to determine a sequence of cubes in $\R^d$ on which we replicate rescaled constructions based on Proposition~\ref{p:expgrowth}. We denote by $Q_n$ a cube of side-length $\lambda_n$ (both the location of the cubes and the side-lengths are to be determined), and we denote by $\tilde Q_n$ the cube with the same center as $Q_n$ and side-length $7\lambda_n$. We will make sure that $\{\tilde Q_n\}$ is a disjoint family contained in a bounded set and it clusters to a point. 

On every $Q_n$ and $\tilde Q_n$ we replicate the construction of the velocity field $u_n$ in Proposition~\ref{p:expgrowth} (we make explicit the dependence of $u_n$ on the index $n$, since the velocity field in Proposition~\ref{p:expgrowth} depends on the initial data), rescaling in space by a factor $\lambda_n$ and in time by a factor $\tau_n$ (which is also to be determined). We neglect a rigid motion, needed to make the cube $Q_n$ concentric and aligned with the cube
$\Omega_0$ in Proposition~\ref{p:expgrowth}, which is irrelevant to compute all needed norms of velocity field and advected scalar. Then we can define the velocity field as a rescaling of the vector field $u_n$ in Proposition~\ref{p:expgrowth}, namely
\begin{equation}\label{e:piecefield}
v_n(x,t) = \frac{\lambda_n}{\tau_n} \, u_n \left(\frac{x}{\lambda_n},\frac{t}{\tau_n}\right)\,,
\end{equation}
and we observe that $v_n$ is supported in the cube $\tilde{Q}_n$.
Next, we let
$$
v = \sum_{n=1}^\infty v_n\,.
$$
Because the $v_n$ are supported in disjoint cubes, it is straightforward to show that $v$ is divergence-free and that $v$ is $C^1$ in space outside of a point in $\RR^d$, which is given by the limit (in the sense of sets) of the cubes $\tilde{Q}_n$ as $n\to \infty$. By Remark \ref{r:uSmooth}, $v$ can be taken smooth outside of this point.
We let $\rho$ be the unique weak solution in $L^\infty([0,T];L^2_{\text{loc}}(\R^d))$ of the transport equation~\eqref{eq:transport} with advecting field $v$ and initial data $\bar\rho$ (notice that $v$ has compact support).

By a scaling computation (as in Section~3.2 of~\cite{ACMReg19}) and using Remark~\ref{r:uSmooth} we see that
$$
\| v(\cdot,t) \|_{\dot{W}^{r,p}(\R^d)} \lesssim \sum_{n=1}^\infty \frac{\lambda_n^\gamma}{\tau_n},
\qquad \forall\,t>0\,,
$$
where
$$
\gamma = 1-r+\frac{d}{p} > 0 \,.
$$
But, thanks to the bound~\eqref{e:expGrowth} provided by Proposition~\ref{p:expgrowth}, for every $n\in \NN$ we have
$$
\| \nabla \rho(\cdot,t) \|_{L^2(\tilde Q_n)} \geq  \exp\left( \frac{\alpha t}{\tau_n} - \beta\right) M_n,
\qquad \forall\,t>0\,,
$$
where we have set 
$$
M_n = \| \nabla \bar\rho \|_{L^2(Q_n)} \,,
$$
Therefore, using the fact that we will select the cubes $\tilde Q_n$ to be disjoint, it follows that
$$
\| \nabla \rho(\cdot,t) \|_{L^2(\R^d)} \gtrsim \sum_{n=1}^\infty \exp\left( \frac{\alpha t}{\tau_n} \right) M_n,
\qquad \forall\,t>0\,.
$$

We conclude that our task is to determine the location of the disjoint cubes $Q_n$ and choose the sequences $\{\lambda_n\}$ and $\{\tau_n\}$ in such a way that
\begin{equation}\label{e:solution1}
\sum_{n=1}^\infty e^{t/\tau_n} M_n = \infty, \qquad \forall\,t>0,
\end{equation}
and
\begin{equation}\label{e:field1}
\sum_{n=1}^\infty \frac{\lambda_n^\gamma}{\tau_n} < \infty, \qquad \forall\,\gamma>0 \,.
\end{equation}

{\em Step 2. Choice of the cubes.} We set $f = | \nabla \bar\rho |^2 \in L^1_{\loc}(\R^d)$, which clearly entails $M_n = \| f \|_{L^1(Q_n)}^{1/2}$. We set 
$$
A_r(x) = \frac{1}{|\mathcal Q_r(x)|} \int_{\mathcal Q_r(x)} f(y) \, dy \,,
$$
where we denote by $\mathcal Q_r(x)$ the cube of side-length $r>0$ centered at $x\in\R^d$, and we set
$$
\tilde D = \left\{ x \in \R^d \;:\; \exists \, \lim_{r \downarrow 0} A_r(x) = f(x) \right\}\,.
$$
By the Lebesgue differentiation theorem we have  $|\R^d \setminus \tilde D|=0$. The assumption that $\bar\rho$ is not a constant function translates into $f \not \equiv 0$, which in turn guarantees the existence of $\bar\delta>0$ and of a bounded set $D \subset \tilde D$, with $|D|>0$, such that
$$
\forall \, x\in D, \qquad \exists \, \lim_{r \downarrow 0} A_r(x) = f(x) \geq \bar \delta > 0\,.
$$
This means that, for every $x\in D$, there exists $\bar r_x>0$ with the property:
$$
\int_{\mathcal Q_r(x)} f(y) \, dy \geq \frac{\bar\delta}{2}  r^d,
\qquad \forall \, 0<r\leq \bar r_x \,.
$$

We can therefore iteratively pick a monotonic sequence $\{\lambda_n\}$ satisfying
\begin{equation}\label{e:lambdaexp}
0 < \lambda_n \leq e^{-n},  \qquad \lambda_n \downarrow 0,
\end{equation} 
and choose the centers $x_n\in D$ of the cubes in such a way that
the cubes $\mathcal Q_{7\lambda_n}(x_n)$ are disjoint and, setting $Q_n = \mathcal Q_{\lambda_n}(x_n)$, we have
\begin{equation}\label{e:Mlarge}
M_n \geq C \lambda_n^{d/2},
\qquad \forall \, n\,.
\end{equation}
The existence of the sequences $\{x_n\}$ and $\{\lambda_n\}$ as above is guaranteed by the fact that 
%
%
we can inductively choose $x_n$ and $\lambda_n>0$ (small enough) to have
$$
\left| D \setminus \bigcup_{k=1}^n \mathcal Q_{7\lambda_k}(x_k) \right| > 0,
\qquad \forall\, n\,.
$$
The fact that $D$ has been chosen to be bounded guarantees that $\{x_n\}$ can be chosen to be a convergent sequence, and $\{\mathcal Q_{7\lambda_n}(x_n)\}$ to be contained in a bounded set.
We conclude that $\{Q_n\}$ is our desired sequence of cubes.

{\em Step 3. Choice of the sequence $\tau_n$ and conclusion.} The lower bound~\eqref{e:Mlarge} shows that the condition~\eqref{e:solution1} for the loss of regularity of the solution holds if
\begin{equation}\label{e:solution2}
\sum_{n=1}^\infty e^{t/\tau_n} \lambda_n^{d/2} = \infty, \qquad \forall\,t>0\,.
\end{equation}
We recall condition~\eqref{e:field1} for the regularity of the velocity field:
\begin{equation}\label{e:field2}
\sum_{n=1}^\infty \frac{\lambda_n^\gamma}{\tau_n} < \infty, \qquad \forall\,\gamma>0\,.
\end{equation}
The sequence $\{\lambda_n\}$ has been implicitly chosen in the previous step to satisfy \eqref{e:lambdaexp}. We now show how it is possible to choose the sequence $\{\tau_n\}$ in such a way that~\eqref{e:solution2} and~\eqref{e:field2} hold. To this end, we set
$$
\tau_n = \left( \log \frac{1}{\lambda_n} \right)^{-2} \,.
$$
The series in condition~\eqref{e:solution2} becomes
\begin{align*}
  \sum_{n=1}^\infty e^{t/\tau_n} \lambda_n^{d/2} 
    &= \sum_{n=1}^\infty \left( e^{\log \frac{1}{\lambda_n}} \right)^{t \log \frac{1}{\lambda_n}} \lambda_n^{d/2} 
  \\
    &= \sum_{n=1}^\infty \left( \frac{1}{\lambda_n} \right)^{t \log \frac{1}{\lambda_n}} \lambda_n^{d/2} 
  = \sum_{n=1}^\infty \lambda_n^{t \log \lambda_n + d/2}\,,
\end{align*}
which diverges since $\lambda_n^{t \log \lambda_n + d/2} \to +\infty$ as $n\to \infty$ for every $t>0$. 

On the other hand, choosing $N=N(\gamma)$ so that 
$$
\left( \log \frac{1}{\lambda_n} \right)^2 \leq \left( \frac{1}{\lambda_n} \right)^{\gamma/2},
\qquad \forall \, n \geq N(\gamma)
$$
(recall that $\lambda_n \downarrow 0$), the series in condition~\eqref{e:field2} can be estimated using~\eqref{e:lambdaexp} as follows:
\begin{align*}
\sum_{n=1}^\infty \frac{\lambda_n^\gamma}{\tau_n}
&= \sum_{n=1}^\infty \left( \log \frac{1}{\lambda_n} \right)^2 \lambda_n^\gamma 
\leq \sum_{n=1}^{N(\gamma)-1} \left( \log \frac{1}{\lambda_n} \right)^2 \lambda_n^\gamma
+ \sum_{n = N(\gamma)}^\infty \lambda_n^{\gamma/2} \\
&\leq \sum_{n=1}^{N(\gamma)-1} \left( \log \frac{1}{\lambda_n} \right)^2 \lambda_n^\gamma
+ \sum_{n = N(\gamma)}^\infty e^{-\gamma n/2} \,,
\end{align*}
which is finite for any $\gamma>0$. This concludes the proof of the  theorem.
\end{proof}

\section{Conclusion} \label{s:end}

In this work, we study properties of weak solutions to a linear transport equation, when the advecting velocity is rough, i.e.,  it has only Sobolev regularity in space.

We extend the results in \cite{ACMReg19} to show that, given any non-constant initial data with square integrable derivative, it is possible to choose the advecting vector field in such a way that the solution loses its regularity instantaneously. To be more precise,  we measure the regularity of the passive scalar in Sobolev spaces and show that all derivatives of the solution of order greater or equal to $1$ blow up in $L^2$ for any $t>0$. This result shows severe ill-posedness in the sense of Hadamard for the transport equation in Sobolev spaces. This result is sharp in the scale of Sobolev spaces, that is, the vector field in our example belongs to all Sobolev spaces that do not embed in the Lipschitz class.

Although the construction is not as explicit as in \cite{ACMReg19}, this example is based on a judicious choice of shear flows acting on the torus, then extended to the full space. Our construction is not universal, in the sense that the advecting field depends on the choice of initial data. It is an open question whether one can construct one single vector field that make the norm of derivatives of the solution blow up for (almost) all initial data. Even though the vector field depends in a strong way on the initial data, the blow-up mechanism described in this work is distinctively linear, since it is based on rescaling and superposing  basic flows and solutions.

\vskip6pt
\enlargethispage{20pt}




\section*{Acknowledgements}

The authors thank Giovanni Alberti for stimulating discussions on loss of regularity for transport equations that led to the problem addressed in this work. They also thank Marco Inversi for a careful reading of the manuscript.




\bibliographystyle{unsrtnat}

\bibliography{DoeringVolume,refs}






\end{document}